\documentclass[10pt]{amsart}

\usepackage{amssymb}
\usepackage{amsthm}
\usepackage{amsmath}

\usepackage[usenames]{color}

\usepackage[foot]{amsaddr}

\usepackage{hyperref}

\theoremstyle{plain}
\newtheorem{proposition}{Proposition}[section]
\newtheorem{theorem}[proposition]{Theorem}
\newtheorem{lemma}[proposition]{Lemma}
\newtheorem{corollary}[proposition]{Corollary}
\theoremstyle{definition}
\newtheorem{example}[proposition]{Example}
\newtheorem{definition}[proposition]{Definition}
\newtheorem{observation}[proposition]{Observation}
\theoremstyle{remark}
\newtheorem{remark}[proposition]{Remark}
\newtheorem{conjecture}[proposition]{Conjecture}

\DeclareMathOperator{\Aut}{Aut}

\DeclareMathOperator{\Id}{Id}

\DeclareMathOperator{\Hol}{Hol}

\DeclareMathOperator{\Euc}{Euc} 
\DeclareMathOperator{\id}{id}

\DeclareMathOperator{\Cc}{\mathcal{C}}

\DeclareMathOperator{\Hc}{\mathcal{H}}

\DeclareMathOperator{\Oc}{\mathcal{O}}
\DeclareMathOperator{\Pc}{\mathcal{P}}

\DeclareMathOperator{\Bb}{\mathbb{B}}
\DeclareMathOperator{\Cb}{\mathbb{C}}
\DeclareMathOperator{\Db}{\mathbb{D}}

\DeclareMathOperator{\Nb}{\mathbb{N}}

\DeclareMathOperator{\Rb}{\mathbb{R}}

\newcommand{\abs}[1]{\left|#1\right|}

\newcommand{\norm}[1]{\left\|#1\right\|}
\newcommand{\wt}[1]{\widetilde{#1}}
\newcommand{\wh}[1]{\widehat{#1}}
\newcommand{\ip}[1]{\left\langle #1\right\rangle}

\begin{document}

\title{Smoothly bounded domains covering compact manifolds}
\author{Andrew Zimmer}\address{Department of Mathematics, Louisiana State University, Baton Rouge, LA 70803}
\email{amzimmer@lsu.edu}
\date{\today}

\begin{abstract} We show that if a bounded domain in complex Euclidean space with $\Cc^{1,1}$ boundary covers a compact manifold, then the domain is biholomorphic to the unit ball.
\end{abstract}

\maketitle

\section{Introduction}

Given a domain $\Omega \subset \Cb^d$ let $\Aut(\Omega)$ denote the biholomorphism group of $\Omega$. When $\Omega$ is bounded, H. Cartan proved that $\Aut(\Omega)$ is a Lie group (with possibly infinitely many connected components) and acts properly on $\Omega$. 

We say that a domain $\Omega \subset \Cb^d$ \emph{covers a compact manifold} if there exists a discrete group $\Gamma \leq \Aut(\Omega)$ such that $\Gamma$ acts freely, properly discontinuously, and co-compactly on $\Omega$. The simplest example of a domain which covers a compact manifold is the unit ball $\Bb_d \subset \Cb^d$. In this case, $\Aut(\Bb_d)$ is isomorphic to the matrix group ${ \rm PU}(1,d)$ and any co-compact torsion free lattice $\Gamma \leq \Aut(\Bb_d)$ acts freely, properly discontinuously, and co-compactly on $\Bb_d$. 

In this paper we prove that, up to biholomorphism, the unit ball is the only domain covering a compact manifold with $\Cc^{1,1}$ boundary. 
 
\begin{theorem}\label{thm:main} Suppose $\Omega \subset \Cb^d$ is a bounded domain which covers a compact manifold. If $\partial \Omega$ is $\Cc^{1,1}$, then $\Omega$ is biholomorphic to the unit ball. 
\end{theorem}

A bounded domain $\Omega \subset \Cb^d$ is called \emph{symmetric} if $\Aut(\Omega)$  is a semisimple Lie group which acts transitively on $\Omega$. A theorem of Borel~\cite{B1963} says that every bounded symmetric domain covers a compact manifold and so we have the following corollary. 

\begin{corollary}\label{cor:BSD} Suppose $\Omega \subset \Cb^d$ is a bounded symmetric domain and $\partial \Omega$ is $\Cc^{1,1}$, then $\Omega$ is biholomorphic to the unit ball. \end{corollary}

Theorem~\ref{thm:main} extends a classical result of Rosay and Wong from the 1970's.

\begin{theorem}[Rosay~\cite{R1979}, Wong~\cite{W1977}]\label{thm:WR} Suppose $\Omega \subset \Cb^d$ is a bounded domain which covers a compact manifold. If $\partial \Omega$ is $\Cc^{2}$, then $\Omega$ is biholomorphic to the unit ball. 
\end{theorem}

\begin{remark} Wong proved Theorem~\ref{thm:WR} for strongly pseudoconvex domains and Rosay extended the result to any bounded domain with $\Cc^2$ boundary. \end{remark}

Over the last forty years many different proofs of Theorem~\ref{thm:WR} have been found, but they all rely on essentially the same idea: every bounded domain $\Omega \subset \Cb^d$ with $\Cc^2$ boundary has at least one strongly pseudoconvex boundary point and the interior complex geometry of $\Omega$ near a strongly pseudoconvex boundary point is close to the interior complex geometry of the unit ball. Then, since $\Omega$ covers a compact manifold, the interior complex geometry of $\Omega$ is everywhere close to the interior complex geometry of the unit ball.  Then a limiting argument shows that $\Omega$ is biholomorphic to the ball. 

One way to make this precise is to consider the Bergman metric $g$ on $\Omega$. This is a $\Aut(\Omega)$-invariant K\"ahler metric on $\Omega$ and, since the boundary is $\Cc^2$,  also complete by a result of Ohsawa~\cite{O1981}. Kim-Yu~\cite{KY1996} proved that the holomorphic sectional curvature of $g$ limits to $-4/(d+1)$ at $\xi_0$ (see also~\cite{K1978}). Since $\Aut(\Omega)$ acts co-compactly on $\Omega$, for any point $z \in \Omega$ there exists a sequence $\varphi_n \in \Aut(\Omega)$ such that $\varphi_n(z) \rightarrow \xi_0$. Then, by the invariance of $g$, the holomorphic curvature at $z$ equals $-4/(d+1)$. Since $z$ was arbitrary, $(\Omega, g)$ has constant holomorphic curvature and hence, by a theorem of Q.K. Lu~\cite{Lu1966}, $\Omega$ is biholomorphic to the ball. For more details, see Section 5 in~\cite{KY1996}.

In the $\Cc^{1,1}$ case it is no longer possible to simply localize around a strongly pseudoconvex point which makes the argument much more complicated. 

\subsection{A conjecture} Recently we generalized Theorem~\ref{thm:WR} in a different direction by only assuming that the domain covers a finite volume manifold. 

\begin{theorem}[Z.~\cite{Z2019_JDG}]Suppose $\Omega \subset \Cb^d$ is a bounded pseudoconvex domain with $\Cc^{2}$ boundary and $\Gamma \leq \Aut(\Omega)$ is a discrete group acting freely on $\Omega$. If $\Gamma \backslash \Omega$ has finite volume with respect to either the Bergman volume, the K{\"a}hler-Einstein volume, or the Kobayashi-Eisenman volume, then $\Omega$ is biholomorphic to the unit ball. 
\end{theorem}

Based on this it seems natural to ask if Theorem~\ref{thm:main} can also be extended to the finite volume case. 

\begin{conjecture} Suppose $\Omega \subset \Cb^d$ is a bounded pseudoconvex domain with $\Cc^{1,1}$ boundary and $\Gamma \leq \Aut(\Omega)$ is a discrete group acting freely on $\Omega$. If $\Gamma \backslash \Omega$ has finite volume with respect to either the Bergman volume, the K{\"a}hler-Einstein volume, or the Kobayashi-Eisenman volume, then $\Omega$ is biholomorphic to the unit ball. 

\end{conjecture}

\subsection{Outline of the proof of Theorem~\ref{thm:main}} Our proof in the $\Cc^{1,1}$ case is very different from the standard proofs in the $\Cc^2$ case and requires  both local and global arguments. 

Fix a bounded domain  $\Omega \subset \Cb^d$ with $\Cc^{1,1}$ boundary and a discrete group $\Gamma \leq \Aut(\Omega)$ such that $\Gamma$ acts freely, properly discontinuously, and co-compactly on $\Omega$. 

\medskip

\emph{Step 1:} For $\alpha > 0$ define
\begin{align*}
\Pc_{\alpha} = \left\{ (z_1,\dots, z_d) \in \Cb^d : { \rm Re}(z_1) > \alpha \sum_{j=2}^d \abs{z_j}^2 \right\}.
\end{align*}
Notice that $\Pc_{\alpha}$ is biholomorphic to the unit ball. 

We use a rescaling argument to show that $\Omega$ is biholomorphic to a domain $D \subset \Cb^d$ where 
\begin{enumerate}
\item there exists $0 < \beta < \alpha$ such that 
\begin{align*}
\Pc_\alpha \subset D \subset \Pc_{\beta}
\end{align*}
\item $\Aut(D)$ contains the one-parameter subgroup 
\begin{align*}
u_t(z_1,\dots,z_d) = (z_1+it,z_2,\dots, z_d).
\end{align*}
\end{enumerate}
 In particular, $\Aut(\Omega) \cong \Aut(D)$ is non-discrete. 

\medskip

\emph{Step 2:} Next we use a theorem of Frankel and Nadel to deduce that $\Omega$ is a bounded symmetric domain. 

\begin{theorem}[{Frankel, Nadel~\cite{F1995, N1990}}]\label{thm:Frankel} Suppose $M$ is a compact complex manifold with $c_1(M) < 0$ and $\wt{M}$ is the universal cover of $M$. If $\Aut\left(\wt{M}\right)$ is non-discrete, then $\wt{M}$ is biholomorphic to either 
\begin{enumerate}
\item  a bounded symmetric domain, or 
\item a non-trivial product $D_1 \times D_2$ where $D_1$ is a bounded symmetric domain and $\Aut(D_2)$ is discrete. 
\end{enumerate} 
\end{theorem}

\begin{remark} \ \begin{enumerate}
\item Nadel~\cite{N1990} proved Theorem~\ref{thm:Frankel} when $d=2$ and then Frankel~\cite{F1995} extended the result to all dimensions. 
\item Theorem~\ref{thm:Frankel} is one of several rigidity results which consider manifolds whose universal cover has a non-discrete isometry group, see~\cite{LW2018,FW2008,E1982}.
\item In our setting, the quotient $\Gamma \backslash \Omega$ will be aspherical and in this special case an alternative proof of Theorem~\ref{thm:Frankel} can be found in~\cite{FW2008}.
\end{enumerate}
\end{remark}

If $M := \Gamma \backslash \Omega$, then $c_1(M) < 0$ (see the discussion on~\cite[pg. 286]{F1995}). Further, the domain $\Omega$ is simply connected (see Proposition~\ref{prop:simply_connected}) and hence is the universal cover of $M$. So by Step 1 and Theorem~\ref{thm:Frankel}, we see that $\Omega$ is either symmetric or biholomorphic a product $D_1 \times D_2$ where $D_1$ is symmetric and $D_2$ has discrete automorphism group. We will use the geometry of the rescaled domain  from Step 1 to show that it is impossible for $\Omega$ to be biholomorphic to such a product. Thus $\Omega$ is a bounded symmetric domain. 

\medskip

\emph{Step 3:} To finish the proof of Theorem~\ref{thm:main} we exploit the geometry of the rescaled domain $D$ and the theory of bounded symmetric domains. Let $\Omega_{HC} \subset \Cb^d$ be the image of the Harish-Chandra embedding of $\Omega$. Then by Step 1, there exists a biholomorphism $F :D \rightarrow \Omega_{HC}$. 

To show that $\Omega$ is biholomorphic to the ball, we introduce the holomorphic function
\begin{align*}
J&:\Db \rightarrow \Cb \\
J&(\lambda) = \det\left(F^\prime\left( \frac{1+\lambda}{1-\lambda},0,\dots,0 \right)\right).
\end{align*}
where $F^\prime(z)$ is the Jacobian matrix of $F$. This function measures the volume contraction/expansion of $F$ along the linear slice 
\begin{align*}
\Cb \cdot e_1 \cap D =\left\{(z,0,\dots,0) \in \Cb^d: {\rm Re}(z) > 0\right\}.
\end{align*}
of $D$. Since $F$ is a biholomorphism, $J$ is nowhere zero. 

We will estimate $J$ using the ``change of variable formula'' for the Bergman kernels on $D$ and $\Omega_{HC}$:
\begin{align*}
 \kappa_D(z,w)=\kappa_{\Omega_{HC}}(F(z),F(w))\det(F^\prime(z))\overline{\det(F^\prime(w))}.
\end{align*}
Combining this with a formula for the Bergman kernel on $\Omega_{HC}$ from~\cite{FK1990}, we show: if $\Omega$ is not biholomorphic to the ball, then $J$ extends continuously to $\partial \Db$ and $J|_{\partial \Db} \equiv 0$. But then the maximal principle would imply that  $J \equiv 0$, which is impossible. So $\Omega$ is biholomorphic to the unit ball. A key part in this step is showing that 
\begin{align*}
\lambda \in \Db \rightarrow F\left( \frac{1+\lambda}{1-\lambda},0,\dots,0 \right) \in \Omega_{HC}
\end{align*}
parameterizes the diagonal of a maximal polydisk in $\Omega_{HC}$. 

\subsection*{Acknowledgements} I would like to thank Gestur Olafsson and Ralf Spatzier for helpful conversations about bounded symmetric domains. This material is based upon work supported by the National Science Foundation under grant DMS-1904099.

\section{Preliminaries}

\subsection{Notation}

For $z_0 \in \Cb^d$ and $r > 0$ let 
\begin{align*}
\Bb_d(z_0;r)  = \{ z \in\Cb^d : \norm{z-z_0} <r\}.
\end{align*}
Also let $e_1, \dots, e_d$ denote the standard basis of $\Cb^d$. 

\subsection{The Kobayashi metric} Given a domain $\Omega \subset \Cb^d$ the \emph{(infinitesimal) Kobayashi metric} is the pseudo-Finsler metric
\begin{align*}
k_{\Omega}(x;v) = \inf \left\{ \abs{\xi} : f \in \Hol(\Delta, \Omega), \ f(0) = x, \ d(f)_0(\xi) = v \right\}.
\end{align*}
By a result of Royden~\cite[Proposition 3]{R1971} the Kobayashi metric is an upper semicontinuous function on $\Omega \times \Cb^d$. In particular if $\sigma:[a,b] \rightarrow \Omega$ is an absolutely continuous curve (as a map $[a,b] \rightarrow \Cb^d$), then the function 
\begin{align*}
t \in [a,b] \rightarrow k_\Omega(\sigma(t); \sigma^\prime(t))
\end{align*}
is integrable and we can define the \emph{length of $\sigma$} to  be
\begin{align*}
\ell_\Omega(\sigma)= \int_a^b k_\Omega(\sigma(t); \sigma^\prime(t)) dt.
\end{align*}
One can then define the \emph{Kobayashi pseudo-distance} to be
\begin{multline*}
 K_\Omega(x,y) = \inf \left\{\ell_\Omega(\sigma) : \sigma\colon[a,b]
 \rightarrow \Omega \text{ is absolutely continuous}, \right. \\
 \left. \text{ with } \sigma(a)=x, \text{ and } \sigma(b)=y\right\}.
\end{multline*}
This definition is equivalent to the standard definition of $K_\Omega$ via analytic chains, see~\cite[Theorem 3.1]{V1989}.

We will use the following property of the Kobayashi metric (which is immediate from the definition). 

\begin{observation}\label{obs:monotone} Suppose $\Omega_1 \subset \Cb^{d_1}$, $\Omega_2 \subset \Cb^{d_2}$ are domains. If $f :\Omega_1 \rightarrow \Omega_2$ is holomorphic, then 
\begin{align*}
K_{\Omega_2}(f(p),f(q)) \leq K_{\Omega_1}(p,q)
\end{align*}
and 
\begin{align*}
k_{\Omega_2}(f(p);d(f)_p(v)) \leq k_{\Omega_1}(p;v)
\end{align*}
for all $p,q \in \Omega_1$ and $v \in \Cb^d$. 
\end{observation}

We will also consider the following special class of maps of the disk into a domain. 

\begin{definition}\label{defn:cplx_geod} Suppose $\Omega \subset \Cb^d$ is a domain. A holomorphic map $\varphi : \Db \rightarrow \Omega$ is called a \emph{complex geodesic} if 
\begin{align*}
K_\Omega(\varphi(z),\varphi(w)) = K_{\Db}(z,w)
\end{align*}
for all $z,w \in \Db$. 
\end{definition}

\subsection{The Bergman kernel and basic properties} Let $\mu$ denote the Lebesgue measure on $\Cb^d$. Then, for a domain $\Omega \subset \Cb^d$ let $\Hc^2(\Omega)$ be the Hilbert space  of holomorphic functions $f:\Omega \rightarrow \Cb$ with $\int_\Omega \abs{f}^2 d\mu<+\infty$. If $\{ \phi_j : j \in J\}$ is an orthonormal basis of $\Hc^2(\Omega)$, then the function
\begin{align*}
\kappa_\Omega & : \Omega \times \Omega \rightarrow \Cb \\
\kappa_\Omega &(z,w) = \sum_{j \in J} \phi_j(z)\overline{\phi_j(w)}
\end{align*}
is called the \emph{Bergman kernel} of $\Omega$.

We now recall two important properties of the Bergman kernel, proofs of both can be found in~\cite[Chapter 12]{JP2013}. 

\begin{proposition}[Monotonicity]\label{prop:bergman_monotone} If $\Omega_1 \subset \Omega_2 \subset \Cb^d$ are domains, then 
\begin{align*}
\kappa_{\Omega_2}(z,z) \leq\kappa_{\Omega_1}(z,z)
\end{align*}
for all $z \in \Omega_1$. 
\end{proposition}

\begin{proposition}[Change of variable formula]\label{prop:change_of_variable} If  $\Omega_1, \Omega_2 \subset \Cb^{d}$ are domains and $F:\Omega_1 \rightarrow \Omega_2$ is a biholomorphism, then
\begin{align*}
\kappa_{\Omega_1}(z,w) =\kappa_{\Omega_2}(F(z), F(w)) \det( F^\prime(z)) \overline{\det( F^\prime(w))}
\end{align*}
for all $z,w \in \Omega_1$. 
\end{proposition}

We will also use the following well known calculation.

\begin{observation}\label{obs:bergman_kernel_parabola} Suppose $\alpha > 0$ and 
\begin{align*}
\Pc_{\alpha} := \left\{ (z_1,\dots, z_d) \in \Cb^d : { \rm Re}(z_1) > \alpha \sum_{j=2}^d \abs{z_j}^2 \right\}.
\end{align*}
Then there exists $C_\alpha > 0$ such that 
\begin{align*}
\kappa_{\Pc_\alpha}\Big((\lambda,0,\dots,0),(z,0,\dots,0)\Big) = C_{\alpha} {\rm Re}(\lambda)^{-(d+1)}
\end{align*}
for all $(\lambda,0,\dots,0) \in \Pc_{\alpha}$.
\end{observation}

Since the proof is short we include it. 

\begin{proof} Let 
\begin{align*}
C_\alpha :=\kappa_{\Omega_1}\Big((1,0,\dots,0),(1,0,\dots,0)\Big)
\end{align*}
and consider the automorphisms $a_t, u_t \in \Aut(\Pc_\alpha)$ given by 
\begin{align*}
a_t(z_1,\dots,z_d) = (e^tz_1,e^{t/2}z_2,\dots, e^{t/2}z_d)
\end{align*}
and 
\begin{align*}
u_t(z_1,\dots,z_d) =(z_1+it,z_2,\dots,z_d).
\end{align*}
Then 
\begin{align*}
(\lambda,0,\dots,0) = u_{{\rm Im}(\lambda)} a_{\log {\rm Re}(\lambda)}(1,0,\dots,0)
\end{align*}
and so Proposition~\ref{prop:change_of_variable} implies that 
\begin{equation*}
\kappa_{\Pc_\alpha}\Big((\lambda,0,\dots,0),(\lambda,0,\dots,0)\Big) = C_{\alpha} {\rm Re}(\lambda)^{-(d+1)}. \qedhere
\end{equation*}

\end{proof}

\subsection{A higher dimensional variant of Hurwitz's theorem}

We will use the following higher dimensional variant of Hurwitz's theorem.

\begin{theorem}[{Deng-Guan-Zhang~\cite[Theorem 2.2]{DGZ2012}}]\label{thm:hurwitz} Suppose that $D \subset \Cb^d$ is a bounded domain and $x\in D$. Let $f_n: D \rightarrow \Cb^d$ be a sequence of injective holomorphic maps such that $f_n(x) = 0$ for all $n$ and $f_n$ converges locally uniformly to a map $f:D \rightarrow \Cb^d$. If there exists $\epsilon > 0$ such that $\Bb_d(0;\epsilon) \subset f_n(\Omega)$ for all $n$, then $f$ is injective. 
\end{theorem}

\section{Domains with co-compact automorphism groups}

In this subsection we prove two basic facts about domains whose automorphism group acts co-compactly, that is there is a compact subset whose translates by the automorphism group cover the domain. Both are probably well known. 

\begin{proposition}\label{prop:proper_metric} If $\Omega \subset \Cb^d$ is a bounded domain and $\Aut(\Omega)$ acts co-compactly on $\Omega$, then $(\Omega, K_\Omega)$ is a proper metric space. Hence, $\Omega$ is pseudoconvex. 
\end{proposition}

\begin{remark} \ \begin{enumerate}
\item Recall, a metric space is called \emph{proper} if bounded sets are relatively compact.  Proper metric spaces are clearly Cauchy complete and so the ``hence'' part of Proposition~\ref{prop:proper_metric} follows from a result of Wu~\cite[Theorem F]{W1967}. 
\item Siegel~\cite{S2008} proved that if a bounded domain covers a compact manifold, then the domain is pseudoconvex (see~\cite[Section 2.1]{SV2018} for an exposition).  
\end{enumerate}
\end{remark}

\begin{proposition}\label{prop:simply_connected} If $\Omega \subset \Cb^d$ is a bounded domain, $\Aut(\Omega)$ acts co-compactly on $\Omega$, and $\partial \Omega$ is $\Cc^1$, then for every $m \geq 1$ the $m^{th}$ homotopy group $\pi_m(\Omega)$ is trivial. In particular, $\Omega$ is simply connected. 
\end{proposition}

\begin{remark} The proof of Proposition~\ref{prop:simply_connected} is a simple modification of the proof of the Lemma on pg. 256 in~\cite{W1977} which in~\cite{W1977} is attributed to R. Greene. 
\end{remark}

\subsection{Proof of Proposition~\ref{prop:proper_metric}} Before proceeding, we recall some terminology. If $(X,d)$ is a metric space, $[a,b] \subset \Rb$, and $\sigma : [a,b] \rightarrow X$ is continuous, then we define the \emph{length of $\sigma$} to be 
\begin{align*}
\ell(\sigma) = \sup\left\{ \sum_{j=1}^N d(\sigma(t_j), \sigma(t_{j+1})) : N \geq 1, \ a \leq t_1 < t_2 < \dots < t_N \leq b \right\}.
\end{align*}
Then $(X,d)$  is called a \emph{length space} if 
\begin{align*}
d(x,y) = \inf\left\{ \ell(\sigma) : \sigma : [0,1] \rightarrow X \text{ continuous with } \sigma(0)=x, \sigma(1)=y\right\}
\end{align*}
for every $x,y \in X$.

We will use the following version of the Hopf-Rinow Theorem (for a proof, see for instance  \cite[Chapter I, Theorem 2.2]{B1995}). 

\begin{theorem}[Hopf--Rinow]\label{thm:hopf_rinow} Suppose $(X,d)$ is a locally compact length metric space. Then the following are equivalent:
\begin{enumerate}
\item $(X,d)$ is a proper metric space,
\item $(X,d)$ is Cauchy complete.
\end{enumerate}
\end{theorem}

We will also use the following lemma. 

\begin{lemma}\label{lem:hopf_rinow} Suppose $(X,d)$ is a locally compact metric space and there exists a compact set $K \subset X$ such that $X = \operatorname{Isom}(X,d) \cdot K$. Then $(X,d)$ is Cauchy complete. 
\end{lemma}

\begin{proof}
We first claim that there exists $\delta >0$ such that for any $x \in X$ then set 
\begin{align*}
\overline{B}_{X}(x;\delta):=\{ y \in X : d(x,y) \leq \delta\}
\end{align*}
is compact. Since $(X,d)$ is locally compact, for any $k \in K$ there exists $\delta_k >0$ such that $\overline{B}_X(k;\delta_k)$ is compact. Then since 
\begin{align*}
K \subset \cup_{k \in K} \{ y \in X : d(k,y) < \delta_k/2\}
\end{align*}
there exists $k_1, \dots, k_N \in K$ such that 
\begin{align*}
K \subset \cup_{i=1}^N \{ y \in X : d(k_i,y) < \delta_{k_i}/2\}.
\end{align*}
Then if 
\begin{align*}
\delta:= \min \{ \delta_{k_i}/2 : i=1,\dots, N\}
\end{align*}
 we see that $\overline{B}_X(x;\delta)$ is compact for any $x \in K$. Since $X = \operatorname{Isom}(X,d) \cdot K$, then $\overline{B}_X(x;\delta)$ is compact for any $x \in X$. 

Now suppose that $x_n$ is a Cauchy sequence in $(X,d)$. Then there exists $N >0$ such that $d(x_n, x_N) < \delta$ for $n \geq N$. So $\{x_n : n \geq N\} \subset \overline{B}_X(x_N; \delta)$. But then there exists a subsequence $x_{n_k}$ which converges. Thus $(X,d)$ is Cauchy complete. 
\end{proof}

\begin{proof}[Proof of Proposition~\ref{prop:proper_metric}] By construction, $(\Omega,K_\Omega)$ is a length metric space. Further, since $\Omega$ is bounded, it is easy to show that $(\Omega, K_\Omega)$ is locally compact.

By Lemma~\ref{lem:hopf_rinow} the metric space $(\Omega, K_\Omega)$ is Cauchy complete. So by Theorem~\ref{thm:hopf_rinow}, $K_\Omega$ is a proper metric on $\Omega$. 

Since $(\Omega, K_\Omega)$ is Cauchy complete, a result of Royden~\cite[Corollary pg. 136]{R1971} says that $\Omega$ is taut. Then $\Omega$ is pseudoconvex by a result of Wu~\cite[Theorem F]{W1967}.
\end{proof}

\subsection{Proof of Proposition~\ref{prop:simply_connected}}

Without loss of generality, we can suppose that $0 \in \Omega$. Then by rotating and scaling we can assume that  
\begin{align*}
1 = \max\{ \norm{z} : z \in \partial \Omega\}
\end{align*}
and $e_1 \in \partial \Omega$. Then define the function 
\begin{align*}
f & : \overline{\Omega} \rightarrow \Cb\\
f& (z_1,\dots,z_d) = e^{z_1-1}.
\end{align*}
Then $\abs{f(z)}\leq 1$ for all $z \in \overline{\Omega}$ with equality if and only if $z =e_1$. 

Now fix a sequence $p_n \in \Omega$ with $e_1 = \lim_{n \rightarrow \infty} p_n$. Since $\Aut(\Omega)$ acts co-compactly on $\Omega$, there exist sequences $\varphi_n \in \Aut(\Omega)$ and $k_n \in \Omega$ such that:
\begin{enumerate}
\item $p_n = \varphi_n(k_n)$
\item $\{ k_n : n \geq 0\}$ is relatively compact in $\Omega$. 
\end{enumerate}
By Montel's theorem we can pass to a subsequence so that $\varphi_n$ converges locally uniformly to a holomorphic map $\varphi_\infty : \Omega \rightarrow \overline{\Omega}$. By passing to another subsequence, we can also assume that $k_n \rightarrow k \in \Omega$. Then 
\begin{align*}
\varphi_\infty(k) =\lim_{n \rightarrow \infty} \varphi_n(k_n) = \lim_{n \rightarrow \infty} p_n = e_1.
\end{align*}
Then consider $g = f\circ \varphi_\infty$, then $\abs{g(z)} \leq 1$ for all $z \in \Omega$ and $\abs{g(k)}=1$. Thus, by the maximal principle, $g \equiv 1$. So $\varphi_\infty \equiv e_1$. 

Since $\partial \Omega$ is $\Cc^1$ there exists a neighborhood $\Oc$ of $e_1$ such that $\Omega \cap \Oc$ is contractible. Now fix a continuous map $\sigma: \mathbb{S}^m \rightarrow \Omega$. Since $\varphi_n$ converges locally uniformly to $\varphi_\infty \equiv e_1$ and $\sigma(\mathbb{S}^m)$ is compact, there exists some $n \geq 0$ such that 
\begin{align*}
(\varphi_n\circ\sigma)(\mathbb{S}^m) \subset \Oc \cap \Omega.
\end{align*}
So $(\varphi_n\circ\sigma)$ is homotopically trivial. So $\sigma$ is  homotopically trivial. Since $\sigma$ was an arbitrary map, $\pi_m(\Omega)=1$.

\section{Rescaling}

As before, for $\alpha > 0$ define
\begin{align*}
\Pc_{\alpha} = \left\{ (z_1,\dots, z_d) \in \Cb^d : { \rm Re}(z_1) > \alpha \sum_{j=2}^d \abs{z_j}^2 \right\}.
\end{align*}

\begin{proposition}\label{prop:rescaling} Suppose $\Omega \subset \Cb^d$ is a bounded domain with $\Cc^{1,1}$ boundary. If $\Aut(\Omega)$ acts co-compactly on $\Omega$, then $\Omega$ is biholomorphic to a domain $D \subset \Cb^d$ where 
\begin{enumerate}
\item $\Pc_\alpha \subset D \subset \Pc_{\beta}$ for some $0 < \beta < \alpha$,
\item $\Aut(D)$ contains the one-parameter subgroup 
\begin{align*}
u_t(z_1,\dots,z_d) = (z_1+it,z_2,\dots, z_d).
\end{align*}
\end{enumerate}
\end{proposition}

\subsection{Rescaling Euclidean balls} Before proving Proposition~\ref{prop:rescaling} we describe a rescaling procedure.

We begin by recalling the definition of the local Hausdorff topology on the set of all convex domains in $\Cb^d$. First, define the \emph{Hausdorff distance} between two compact sets $A,B \subset \Cb^d$  by
\begin{align*}
d_{H}(A,B) = \max\left\{ \max_{a \in A} \min_{b \in B} \norm{a-b}, \max_{b \in B} \min_{a \in A} \norm{b-a} \right\}.
\end{align*}
To obtain a topology on the set of all convex domains in $\Cb^d$, we consider the \emph{local Hausdorff pseudo-distances} defined by
\begin{align*}
d_{H}^{(R)}(A,B) = d_{H}\left(A \cap \overline{\Bb_d(0;R)},B\cap \overline{\Bb_d(0;R)}\right), \ R >0.
\end{align*}
Then a sequence of convex domains $\Omega_n$ converges to a convex domain $\Omega$ if there exists some $R_0 \geq 0$ such that 
\begin{align*}
\lim_{n \rightarrow \infty} d_{H}^{(R)}\left(\overline{\Omega_n},\overline{\Omega}\right) = 0
\end{align*}
for all $R \geq R_0$. 

The Kobayashi distance is continuous with respect to this topology, see for instance~\cite[Theorem 4.1]{Z2016}. 

\begin{theorem}\label{thm:kob_conv} Suppose $\Omega_n \subset \Cb^d$ is a sequence of convex domains and $\Omega = \lim_{n \rightarrow \infty} \Omega_n$ in the local Hausdorff topology. Assume the Kobayashi metric is non-degenerate on $\Omega$ and each $\Omega_n$. Then 
\begin{align*}
K_\Omega(p,q) = \lim_{n \rightarrow \infty} K_{\Omega_n}(p,q)
\end{align*}
for all $p, q\in \Omega$. Moreover, the convergence is uniform on compact subsets of $\Omega \times \Omega$. 
\end{theorem}

\begin{remark} Under the hypothesis of Theorem~\ref{thm:kob_conv}: if $K \subset \Omega$ is a compact set, then $K \subset \Omega_n$ for $n$ sufficiently large (see ~\cite[Lemma 4.4]{Z2016}). Thus, $ K_{\Omega_n}(p,q)$ is well defined for $n$ sufficiently large (which depends on $p,q$). 
\end{remark}

We end this discussion with the following example.

\begin{example}\label{ex:rescaling_balls}
Fix $r > 0$, a sequence $r_n>0$ converging to $0$, and the sequence of linear maps 
\begin{align*}
\Lambda_n(z_1,\dots,z_d) = \left( \frac{1}{r_n} z_1, \frac{1}{\sqrt{r_n}} z_2, \dots \frac{1}{\sqrt{r_n}} z_d \right). 
\end{align*}
Then 
\begin{align*}
\Pc_{1/(2r)} = \lim_{n \rightarrow \infty} \Lambda_n \Bb_d(re_1;r)
\end{align*}
in the local Hausdorff topology. 
\end{example}

\subsection{The proof of Proposition~\ref{prop:rescaling}}

The rest of the section is devoted to the proof of the Proposition. So suppose that  $\Omega$ is a bounded domain with $\Cc^{1,1}$ boundary and  there exists a compact set $K \subset \Omega$ such that $\Aut(\Omega) \cdot K = \Omega$.

\begin{lemma} After applying an affine transformation, we can assume that
\begin{align*}
\Bb_d(re_1; r)\subset \Omega \subset \Bb_d(e_1;1)
\end{align*}
for some $r \in (0,1)$.
\end{lemma}

\begin{proof}
By translating we can assume that $e_1 \in \Omega$. Then pick $\xi_0 \in \partial \Omega$ such that 
\begin{align*}
\norm{\xi_0-e_1} = \max\{ \norm{\xi-e_1} : \xi\in \partial \Omega\}.
\end{align*} 
By rotating and scaling $\Omega$ about $e_1$, we can assume that $\xi_0=0$. Then 
\begin{align*}
 \Omega \subset \Bb_d(e_1;1).
\end{align*}

For $\xi \in \partial \Omega$, let $n_\Omega(\xi)$ be the inward pointing normal unit vector at $\xi$. Since $\partial\Omega$ is $\Cc^{1,1}$ there exists some $r > 0$ such that 
\begin{align*}
\Bb_d(\xi+rn_\Omega(\xi); r) \subset \Omega
\end{align*}
for every $\xi \in \partial \Omega$. Then, since $\Omega \subset \Bb_d(e_1;1)$, we have $n_\Omega(0)=e_1$ and so
\begin{equation*}
\Bb_d(re_1; r) \subset \Omega. \qedhere
\end{equation*}
\end{proof}

Fix a sequence $r_n \in (0,r)$ converging to $0$. Then pick $\varphi_n \in \Aut(\Omega)$ and $k_n \in K$ such that $\varphi_n(k_n)=r_ne_1$. Then consider the dilations 
\begin{align*}
\Lambda_n(z_1,\dots,z_d) = \left( \frac{1}{r_n} z_1, \frac{1}{\sqrt{r_n}} z_2, \dots \frac{1}{\sqrt{r_n}} z_d \right). 
\end{align*}
Let $\Omega_n := \Lambda_n \Omega$ and $F_n: = \Lambda_n \varphi_n: \Omega \rightarrow \Omega_n$. Then 
\begin{align*}
\Lambda_n \Bb_d(re_1;r) \subset \Omega_n \subset  \Lambda_n \Bb_d(e_1;1).
\end{align*}
Further, by Example~\ref{ex:rescaling_balls}, 
\begin{align*}
\Pc_\alpha = \lim_{n \rightarrow \infty} \Lambda_n \Bb_d(re_1;r) \text{ where } \alpha : = \frac{1}{2r}
\end{align*}
and
\begin{align*}
\Pc_\beta = \lim_{n \rightarrow \infty} \Lambda_n \Bb_d(e_1;1) \text{ where } \beta : = \frac{1}{2}
\end{align*}
in the local Hausdorff topology.

\begin{lemma} After passing to a subsequence, $F_n$ converges to a holomorphic embedding $F: \Omega \rightarrow \Cb^d$. Moreover, if $D = F(\Omega)$, then 
\begin{align*}
\Pc_{\alpha} \subset D \subset \Pc_{\beta}.
\end{align*}
\end{lemma}

\begin{proof} By construction $F_n(k_n) = e_1$ and Observation~\ref{obs:monotone} implies that
\begin{align*}
K_{\Omega_n}(z,w) \leq K_{\Lambda_n \Bb_d(e_1;1)}(z,w)
\end{align*}
for all $z,w \in \Omega_n$. Theorem~\ref{thm:kob_conv} implies that  
\begin{align*}
K_{\Pc_{\beta}}=\lim_{n \rightarrow \infty} K_{\Lambda_n \Bb_d(e_1;1)}
\end{align*}
locally uniformly.  So, using the Arzel\'a-Ascoli theorem, we can pass to a subsequence where $F_n$ converges locally uniformly to a holomorphic map $F : \Omega \rightarrow \Cb^d$. 

Let $D = F(\Omega)$. Since 
\begin{align*}
\Lambda_n \Bb_d(re_1;r) \subset F_n(\Omega) \subset \Lambda_n \Bb(e_1;1)
\end{align*}
for every $n$ we see that 
\begin{equation}
\label{eq:inclusions}
\overline{\Pc_{\alpha}} \subset D \subset \overline{\Pc_{\beta}}.
\end{equation}

Next we use Theorem~\ref{thm:hurwitz} to show that $F$ is injective. Since 
\begin{align*}
\Pc_{\alpha} = \lim_{n \rightarrow \infty} \Lambda_n \Bb_d(re_1;r)
\end{align*}
in the local Hausdorff topology and $\Lambda_n \Bb_d(re_1;r) \subset \Omega_n$ for every $n \geq 0$, there exists $\epsilon > 0$ such that 
\begin{align*}
\Bb_d(e_1;\epsilon) \subset F_n(\Omega)
\end{align*}
for every $n \geq 0$. By passing to a subsequence we can suppose that $k_n \rightarrow k \in K$. Then consider the maps 
\begin{align*}
G_n(z) = F_n(z) -F_n(k).
\end{align*}
Since 
\begin{align*}
\lim_{n \rightarrow \infty} F_n(k)= \lim_{n \rightarrow \infty} F_n(k_n) = F(k) = e_1,
\end{align*}
$G_n$ converges locally uniformly to $F-e_1$. Further, by passing to another subsequence we can suppose that $\norm{e_1-F_n(k)} < \epsilon/2$ for every $n \geq 0$. Then for every $n \geq 0$, the map $G_n$ is injective, $G_n(k) = 0$, and
\begin{align*}
\Bb_d(0;\epsilon/2) \subset G_n(\Omega).
\end{align*}
So $F$ is injective by Theorem~\ref{thm:hurwitz}. Thus $F$ is an embedding. 

Now since $F$ is an embedding, $D$ is an open set and so Equation~\eqref{eq:inclusions} becomes
\begin{align*}
\Pc_{\alpha} = {\rm int} \left( \overline{\Pc_{\alpha}}\right) \subset D \subset {\rm int} \left(\overline{\Pc_{\beta}}\right) = \Pc_\beta.
\end{align*}
This completes the proof. 
\end{proof}

Showing that $\Aut(D)$ contains a one-parameter subgroup requires some preliminary lemmas.

\begin{lemma}\label{lem:included_in_D} Suppose $(z_n)_{n \geq 0}$ is a sequence, $z_n \in \Omega_n$ for every $n$, $\lim_{n \rightarrow \infty} z_n = z$, and 
\begin{align*}
\liminf_{n \rightarrow \infty} K_{\Omega_n}(e_1, z_n) < +\infty,
\end{align*}
then $z \in D$. 
\end{lemma}

\begin{proof} Fix $z_0 \in \Omega$ and let
\begin{align*}
R = \max_{k \in K} K_\Omega(z_0, k).
\end{align*}
Then pick $n_j \rightarrow \infty$ such that 
\begin{align*}
M:=\lim_{j \rightarrow \infty} K_{\Omega_{n_j}}(e_1, z_{n_j}) < +\infty.
\end{align*}

Since $F_n(k_n)=e_1$ and $k_n \in K$, for each $j \geq 0$, there exists 
\begin{align*}
w_j \in \overline{B_\Omega(z_0; R+M)}
\end{align*}
such that $F_{n_j}(w_j) = z_{n_j}$. By Proposition~\ref{prop:proper_metric},  $K_\Omega$ is a proper metric on $\Omega$. So we can pass to a subsequence such that $w_j \rightarrow w \in \Omega$. Since $F_n \rightarrow F$ locally uniformly, we then have
\begin{align*}
F(w) = \lim_{j \rightarrow \infty} F_{n_j}(w_j) = \lim_{j \rightarrow \infty} z_j = z.
\end{align*}
So $z \in F(\Omega) =D$. 
\end{proof}

Before proceeding we recall some standard notations. First, let $\ip{\cdot, \cdot}$ denote the standard inner product on $\Cb^d$, that is 
\begin{align*}
\ip{z,w} = \overline{w}^t z
\end{align*}
for all $z,w \in \Cb^d$. Then given a $\Cc^1$ function $f : \Cb^d \rightarrow \Rb$ let $\nabla f$ denote the gradient of $f$, that is 
\begin{align*}
\lim_{h \rightarrow 0} \frac{f(z+hv)-f(z)}{h} = {\rm Re} \ip{\nabla f(z), v}
\end{align*}
for all $z,v \in \Cb^d$. 

The next lemma essentially says that the distance to the boundary in the tangential direction is much larger than the distance to the boundary in the normal direction.

\begin{lemma}\label{lem:vertical_to_horizontal}
For every $m > 0$, there exists $\delta_m > 0$ such that:  if $z_0 \in \Omega \cap \Bb_d(0;\delta_m)$, $T > 0$, and 
\begin{align*}
\{ z_0+xe_1: -T < x < T\} \subset \Omega,
\end{align*}
then 
\begin{align*}
\{ z_0+(x+iy)e_1: -T/2 \leq x \leq T/2, \ -mT \leq y \leq mT\ \} \subset \Omega.
\end{align*}
\end{lemma}

\begin{proof} Fix a $\Cc^1$ defining function $\rho : \Cb^d \rightarrow \Rb$ of $\Omega$, i.e. 
\begin{align*}
\Omega = \{ z \in \Cb^d : \rho(z) < 0\}
\end{align*}
 and $\nabla \rho(z) \neq 0$ in a neighborhood of $\partial\Omega$. Since $0 \in \partial \Omega$ and 
 \begin{align*}
 \Omega \subset \Bb_d(e_1;1)
 \end{align*}
 we must have $\nabla \rho(0) = -e_1$. 
 
 Since 
 \begin{align*}
 \{ 0\} = \left\{ z \in \overline{\Omega} : { \rm Re} \ip{z,e_1}=0\right\},
 \end{align*}
 there exists $\delta_m > 0$ such that: if $z \in \Omega$ and ${ \rm Re} \ip{z,e_1} < 2\delta_m$, then 
 \begin{align*}
 \norm{\nabla \rho(z) -(-e_1)} < \frac{1}{2m+1}.
 \end{align*} 
 
Now fix $z_0 \in \Omega \cap \Bb_d(0;\delta_m)$ and $T > 0$ such that
\begin{align*}
\{ z_0+xe_1: -T < x < T\} \subset \Omega.
\end{align*}
Since $\Omega \subset  \Bb_d(e_1;1)$ we have
\begin{align*}
0 \leq { \rm Re} \ip{z_0-Te_1,e_1} = { \rm Re} \ip{z_0,e_1} - T <\delta_m -T.
\end{align*}
So $T \leq \delta_m$. Thus 
\begin{align*}
{ \rm Re} \ip{z_0+xe_1,e_1} < 2\delta_m
\end{align*}
when $-T < x < T$. 

Now fix $-T/2 \leq x \leq T/2$. Then 
\begin{align*}
\rho(z_0+xe_1) &= \rho(z_0-Te_1) + \int_{-T}^x{ \rm Re} \ip{ \nabla \rho(z_0+te_1), e_1} dt \\
& \leq  \rho(z_0-Te_1) +\int_{-T}^x \left(-1+\frac{1}{2m+1}\right) dt\\
& \leq 0- \frac{2m}{2m+1} (x+T) \leq - \frac{m}{2m+1} T.
\end{align*}
Then let
\begin{align*}
y^+ := \min\{ y \geq 0 : z_0+(x+iy)e_1 \in \partial\Omega\}.
\end{align*}
Notice that, if $y \in [0,y^+)$, then $z_0+(x+iy)e_1 \in \Omega$ and 
\begin{align*}
{ \rm Re} \ip{z_0+(x+iy)e_1,e_1}={ \rm Re} \ip{z_0+xe_1,e_1} < 2\delta_m.
\end{align*}
So
\begin{align*}
0=\rho(z_0+(x+iy^+)e_1) &= \rho(z_0+xe_1) + \int_0^{y^+}{ \rm Re} \ip{ \nabla \rho(z_0+(x+iy)e_1), ie_1} dy\\
&  \leq - \frac{m}{2m+1} T+\int_0^{y^+}\frac{1}{2m+1} dy = \frac{1}{2m+1} \left( y^+ - mT \right).
\end{align*}
Hence $y^+ \geq mT$. Next define 
\begin{align*}
y^- = \max\{ y \leq 0 : z_0+(x+iy)e_1 \in \partial\Omega\}.
\end{align*}
Then a similar argument shows that $y^- \leq -mT$. So 
\begin{align*}
\{ z_0 + (x+iy)e_1 : -mT \leq y \leq mT \} \subset \Omega.
\end{align*}
Since $-T/2 \leq x \leq T/2$ was arbitrary, this completes the proof of the lemma. 

\end{proof}

\begin{lemma} $\Aut(D)$ contains the one-parameter subgroup 
\begin{align*}
u_t(z_1,\dots,z_d) = (z_1+it,z_2,\dots, z_d).
\end{align*}
\end{lemma}

\begin{proof}It is enough to fix $w_0 \in D$ and $t \in \Rb$, then show that $w_0+ite_1 \in D$. 

Since the sequence $F_n$ converges locally uniformly to $F$, there exists $\epsilon >0$ and $N \geq 0$ such that 
\begin{align*}
\Bb_d(w_0;\epsilon) \subset F_n(\Omega) = \Lambda_n(\Omega)
\end{align*}
for all $n \geq N$. 

Define $w_n : = \Lambda_n^{-1}w_0$. Then 
\begin{align}\label{eq:w_n_vertical}
\{ w_n + x e_1 : -r_n\epsilon < x <  r_n\epsilon\} \subset \Lambda_n^{-1}\Bb_d(w_0;\epsilon)  \subset \Omega
\end{align}
when $n \geq N$. 

Fix $m \in \Nb$ such that $\abs{t} < m\epsilon$. Let $\delta_m > 0$ be the associated constant from Lemma~\ref{lem:vertical_to_horizontal}. Since $r_n \rightarrow 0$, 
\begin{align*}
\lim_{n \rightarrow \infty} w_n = 0
\end{align*}
so by increasing $N$ we can assume that 
\begin{align*}
w_n \in \Bb_d(0;\delta_m)
\end{align*}
when $n \geq N$. Then by Equation~\eqref{eq:w_n_vertical} and Lemma~\ref{lem:vertical_to_horizontal}
\begin{align*}
\left\{ w_n + (x+iy)e_1 : -r_n\epsilon/2 < x <  r_n\epsilon/2, \ -mr_n\epsilon < y < mr_n\epsilon \right\} \subset\Omega
\end{align*}
when $n \geq N$. So 
\begin{align*}
S:=\left\{ w_0 + (x+iy)e_1 : -\epsilon/2 < x <  \epsilon/2, \ -m\epsilon < y < m\epsilon \right\} \subset\Lambda_n\Omega.
\end{align*}

Then, when $n \geq N$ 
\begin{align*}
K_{\Lambda_n\Omega}(w_0, w_0+ite_1) \leq K_{S}(w_0, w_0+ite_1)
\end{align*}
and so 
\begin{align*}
\sup_{n \geq N} K_{\Lambda_n\Omega}(w_0, w_0+ite_1) < +\infty.
\end{align*}
Hence Lemma~\ref{lem:included_in_D} implies that $w_0 + ite_1 \in D$. 

\end{proof}

\section{The geometry of the rescaled domain}

For the rest of this section suppose that $D \subset \Cb^d$ is a domain where
\begin{align*}
\Pc_\alpha \subset D \subset \Pc_{\beta}
\end{align*}
for some  $0 < \beta < \alpha$.

Define
\begin{align*}
\Hc_D := D \cap \Cb \cdot e_1 =  \left\{ (z,0,\dots, 0) \in \Cb^d : { \rm Re}(z) > 0\right\}
\end{align*}
and 
\begin{align*}
\Hc := \{ z\in \Cb : { \rm Re}(z) > 0\}.
\end{align*}

\begin{observation}\label{obs:dist_equality} If $z,w \in \Hc$, then
\begin{align*}
K_{D}\Big((z,0,\dots,0), (w,0,\dots,0)\Big) = K_{\Hc}(z,w).
\end{align*}
\end{observation}

\begin{proof} The inclusion map 
\begin{align*}
\iota&: \Hc \rightarrow D\\
\iota&(z) = (z,0,\dots, 0)
\end{align*}
implies that 
\begin{align*}
K_{D}\Big((z,0,\dots,0), (w,0,\dots,0)\Big) \leq K_{\Hc}(z,w).
\end{align*}
for all $z,w \in \Hc$. Since 
\begin{align*}
D \subset \Pc_\beta \subset \left\{ (z_1,\dots, z_d) \in \Cb^d : { \rm Im}(z_1) > 0\right\},
\end{align*}
the projection map 
\begin{align*}
\pi&: D \rightarrow \Hc\\
\pi&(z_1,\dots,z_d) = z_1
\end{align*}
implies that 
\begin{align*}
K_{\Hc}(z,w) \leq K_{D}\Big((z,0,\dots,0), (w,0,\dots,0)\Big)
\end{align*}
for all $z,w \in \Hc$. 
\end{proof}

Observation~\ref{obs:dist_equality} implies that $\Hc_D$ can be parametrized to be a complex geodesic (see Definition~\ref{defn:cplx_geod}). The next Observation proves that, up to parametrization, this is the only complex geodesic joining two points in $\Hc_D$.  

\begin{proposition}\label{prop:unique_slice} Suppose $p,q \in \Hc_D$ are distinct and $\varphi : \Db \rightarrow D$ is a complex geodesic with $p,q \in \varphi(\Db)$. Then there exists $\phi \in \Aut(\Db)$ such that 
\begin{align*}
(\varphi \circ \phi)(\lambda) = \left(\frac{1+\lambda}{1-\lambda},0,\dots,0\right)
\end{align*}
for all $\lambda \in \Db$. In particular, $\varphi(\Db) = \Hc_D$. 
\end{proposition}

\begin{proof} By hypothesis
\begin{align*}
p = (p_1,0,\dots,0) \text{ and } q = (q_1,0,\dots,0)
\end{align*}
for some $p_1,q_1 \in \Hc$.

Let $f : \Hc \rightarrow \Db$ be a biholomorphism and consider the map
\begin{align*}
\wh{\varphi} := \varphi \circ f: \Hc \rightarrow D.
\end{align*}
Let $\wh{\varphi}_1, \dots, \wh{\varphi}_d$ denote the coordinate functions of $\wh{\varphi}$. Since 
\begin{align*}
D \subset \Pc_\beta \subset \left\{ (z_1,\dots, z_d) \in \Cb^d : { \rm Re}(z_1) > 0\right\},
\end{align*}
we have $\wh{\varphi}_1(\Hc) \subset \Hc$. Further by Observation~\ref{obs:dist_equality}, if $\wh{\varphi}(\lambda_1) = p$ and $\wh{\varphi}(\lambda_2) = q$, then
\begin{align*}
K_{\Hc}(\lambda_1,\lambda_2) &= K_{\Db}(f(\lambda_1), f(\lambda_2)) = K_D\big(\varphi(f(\lambda_1)), \varphi(f(\lambda_2))\big) =  K_{D}(p,q) \\
& = K_{\Hc}(p_1,q_1) = K_{\Hc}(\wh{\varphi}_1(\lambda_1),\wh{\varphi}_1(\lambda_2)).
\end{align*}
So by the Schwarz lemma, $\wh{\varphi}_1$ is a biholomorphism of $\Hc$. Then by replacing $f$ with $f \circ \wh{\varphi}_1^{-1}$, we can assume that  $\wh{\varphi}_1 = \id$. 

We claim that $\wh{\varphi}_j \equiv 0$ for $2 \leq j \leq d$. Fix $t \in \Rb$, then since $D \subset \Pc_{\beta}$ we have 
\begin{align*}
\limsup_{\lambda \rightarrow it} \sum_{j=2}^d \abs{\wh{\varphi}_j(\lambda)}^2 \leq \limsup_{\lambda \rightarrow it} {\rm Re}\left(\wh{\varphi}_1(\lambda)\right) = \limsup_{\lambda \rightarrow it} {\rm Re}(\lambda)=0.
\end{align*}
So 
\begin{align*}
\lim_{\lambda \rightarrow it} \wh{\varphi}_j(\lambda) = 0
\end{align*}
for $2 \leq j \leq d$. So $\wh{\varphi}_j$ extends continuously to $\Hc \cup i\Rb$ with $\wh{\varphi}_j |_{i\Rb} \equiv 0$. So by the Schwarz reflection principle, $\wh{\varphi}_j$ extends holomorphically to all of $\Cb$. But then, since $\wh{\varphi}_j |_{i\Rb} \equiv 0$, we have $\wh{\varphi}_j \equiv 0$. 

So 
\begin{align*}
\varphi(\lambda) = \left(\varphi_1(\lambda),0,\dots,0\right)
\end{align*}
where $\varphi_1 : \Db \rightarrow \Hc$ is a biholomorphism. Finally, define $\phi \in \Aut(\Db)$ by 
\begin{align*}
\phi(\lambda) = \varphi_1^{-1}\left(\frac{1+\lambda}{1-\lambda}\right).
\end{align*}
Then 
\begin{equation*}
(\varphi \circ \phi)(\lambda) = \left(\frac{1+\lambda}{1-\lambda},0,\dots,0\right)
\end{equation*}
for all $\lambda \in \Db$. 
\end{proof}

\begin{proposition}\label{prop:bd_distance} Suppose $(z_n)_{n\geq 0},(w_n)_{n\geq0}$ are sequences in $D$ with
\begin{align*}
\lim_{n \rightarrow \infty} z_n = \xi \in (i\Rb) \times \{(0,\dots,0)\} = \overline{\Hc}_D \cap \partial D
\end{align*}
and 
\begin{align*}
\limsup_{n \rightarrow \infty} K_D(w_n,z_n) < +\infty,
\end{align*}
then
\begin{align*}
\lim_{n \rightarrow \infty} w_n = \xi.
\end{align*}
\end{proposition}

\begin{proof}
Notice that 
\begin{align*}
K_{\Pc_\beta}(z,w) \leq K_D(z,w)
\end{align*}
for all $z,w \in D$ and 
\begin{align*}
(i\Rb) \times \{(0,\dots,0)\} \subset \partial D \cap \partial \Pc_{\beta}.
\end{align*}
So this proposition follows immediately from the well understood geometry of $(\Pc_\beta, K_{\Pc_{\beta}})$ - it is a standard model of complex hyperbolic $d$-space. 

For the reader's convenience we provide a complete argument. Since $\Pc_{\beta}$ is convex, there exists $H \subset \Cb^d$ a complex affine hyperplane where $H \cap \Pc_{\beta} =\emptyset$ and $\xi \in H$. Since $\Pc_{\beta}$ is strictly convex, $H \cap \partial \Pc_{\beta} = \{\xi\}$. By standard estimates for the Kobayashi distance on a convex domain, see for instance~\cite[Lemma 4.2]{Z2017b}, 
\begin{align*}
K_{\Pc_{\beta}}(z,w) \geq \frac{1}{2} \log \frac{d_{\Euc}(w,H)}{d_{\Euc}(z,H)} 
\end{align*}
for all $z,w \in \Pc_{\beta}$. So we must have 
\begin{align*}
\lim_{n \rightarrow \infty} d_{\Euc}(w_n,H) = 0.
\end{align*}
Then, since $H \cap \partial \Pc_{\beta} = \{\xi\}$, we have $\lim_{n \rightarrow \infty} w_n = \xi$. 
\end{proof}

\section{The domain is symmetric}

In this section we prove the following. 

\begin{proposition}\label{prop:BSD} Suppose $\Omega \subset \Cb^d$ is a bounded domain which covers a compact manifold. If $\partial \Omega$ is $\Cc^{1,1}$, then $\Omega$ is a bounded symmetric domain. 
\end{proposition}

Before starting the proof, we recall the following notation. 

\begin{definition} 
Given a domain $\Omega \subset \Cb^d$, let $\Aut_0(\Omega)$ denote the connected component of the identity in $\Aut(\Omega)$. 
\end{definition}

\begin{proof}[Proof of Proposition~\ref{prop:BSD}] Proposition~\ref{prop:rescaling} implies that $\Aut(\Omega)$ is non-discrete and Proposition~\ref{prop:simply_connected} implies that $\Omega$ is simply connected. Hence by Theorem~\ref{thm:Frankel} either 
\begin{enumerate}
\item $\Omega$ is a bounded symmetric domain
\item  $\Omega$ is biholomorphic to $D_1\times D_2$ where $D_1$ is a bounded symmetric domain and $\Aut(D_2)$ is an infinite discrete group. 
\end{enumerate}
We assume the second possibility and derive a contradiction. 

By Proposition~\ref{prop:rescaling}, there exists a biholomorphism $F:D_1 \times D_2 \rightarrow D$ where $D \subset \Cb^d$ is a domain such that
\begin{align*}
\Pc_\alpha \subset D \subset \Pc_{\beta}
\end{align*}
for some $\alpha > \beta > 0$ and $\Aut(D)$ contains the one-parameter subgroup 
\begin{align*}
u_t(z_1,\dots,z_d) = (z_1+it,z_2,\dots, z_d).
\end{align*}

Define 
\begin{align*}
G : = F \circ \Big(\{\Id\} \times \Aut(D_2) \Big)\circ F^{-1} \leq \Aut(D). 
\end{align*}
Then, by assumption, $G$ is an infinite discrete subgroup of $\Aut(D)$ and $G$ commutes with $\Aut_0(D)$. We will obtain a contradiction by establishing the following. 

\medskip

\noindent \textbf{Claim:} $G$ is a finite group. 

\medskip

Since $\Aut_0(D_1 \times D_2) = \Aut_0(D_1) \times \{\id\}$, for any $z=(z_1,z_2) \in D_1 \times D_2$ we have
\begin{align*}
\Aut_0(D_1 \times D_2) \cdot z = D_1 \times \{z_2\}.
\end{align*}
In particular, the orbit $\Aut_0(D_1 \times D_2) \cdot z$ is a complex analytic variety in $D_1 \times D_2$. Thus for any $w \in D$, the orbit $\Aut_0(D)\cdot w$ is a complex analytic variety in $D$. Further, since $\Aut(D)$ contains the one-parameter group 
\begin{align*}
u_t(z_1,\dots,z_d) = (z_1+it, z_2, \dots, z_d),
\end{align*}
for any $w_0 \in D$ and $w \in \Aut_0(D)\cdot w_0$, the tangent space $T_w( \Aut_0(D)\cdot w_0)$ of $\Aut_0(D)\cdot w_0$ at $w$ contains $ie_1$. Thus, since  $\Aut_0(D)\cdot w_0$ is a complex analytic variety, 
\begin{align}
\label{eq:tangent_planes} 
\Cb \cdot e_1 \subset T_{w} \Big(\Aut_0(D)\cdot w_0\Big).
\end{align}

As before, let $\Hc_D := \{ (z,0,\dots,0) : { \rm Re}(z) > 0\}$. Then by Equation~\eqref{eq:tangent_planes} 
\begin{align*}
\Hc_D \subset \Aut_0(D)\cdot e_1.
\end{align*}
So for each $z \in \Hc_D$, there exists $\phi_z \in \Aut_0(D)$ such that $\phi_z (e_1) = z$. 

Now suppose $g \in G$. Then for $z \in \Hc_D$ we have 
\begin{align*}
K_D(z, g(z)) = K_D(\phi_z(e_1), g\phi_z(e_1)) = K_D(\phi_z(e_1), \phi_zg(e_1))=K_D(e_1, g(e_1))
\end{align*}
since $G$ commutes with $\Aut_0(D)$. Hence 
\begin{align}
\label{eq:bd_distance}
\sup_{z \in \Hc_D} K_D(z,g(z)) =K_D(e_1, g(e_1))< +\infty.
\end{align}

Let $\Hc :=\{\lambda \in \Cb : { \rm Re}(\lambda) > 0\}$ and define the map
\begin{align*}
\psi&: \Hc \rightarrow D \\
\psi&(\lambda) = g(\lambda,0,\dots,0).
\end{align*}
Then let $\psi_1,\dots, \psi_d$ denote the coordinate functions of $\psi$. By Proposition~\ref{prop:bd_distance} and Equation~\eqref{eq:bd_distance}, if $t \in \Rb$, then 
\begin{align*}
\lim_{\lambda \rightarrow it} \psi(\lambda) = (it,0,\dots,0).
\end{align*}
Thus by applying the Schwarz reflection principle to each $\psi_j$, we can extend $\psi$ to a map $\Cb \rightarrow \Cb^d$ such that
\begin{align*}
\psi(it) = (it,0,\dots,0)
\end{align*}
for $t \in\Rb$. But then by the identity theorem for holomorphic functions we have
\begin{align*}
\psi(\lambda) = (\lambda,0,\dots,0)
\end{align*}
for all $\lambda \in \Cb$. In particular, $g(e_1) = e_1$. 

Since $g \in G$ was arbitrary we see that 
\begin{align*}
G \cdot e_1 = e_1.
\end{align*}
Since $D$ is biholomorphic to a bounded domain, $\Aut(D)$ acts properly on $D$ and hence $G$ must be compact. Since $G$ is also discrete, we see that $G$ is finite. Thus we have a contradiction. 
\end{proof} 

\section{Polydisks in bounded symmetric domains} 

In this section we recall some facts about polydisks in bounded symmetric domains. 

\begin{definition} Suppose $\Omega$ is a bounded symmetric domain. The \emph{real rank of $\Omega$} is the largest integer $r$ such that there exists a holomorphic isometric embedding $f:(\Db^r, K_{\Db^r}) \rightarrow (\Omega, K_\Omega)$.
\end{definition}

From the characterization of bounded symmetric domains, every bounded symmetric domain has real rank at least one. Moreover,  the real rank is one if and only if the symmetric domain is biholomorphic to the unit ball. 

The next result says that there are many isometric embeddings of polydisks, see~\cite[pg. 280]{W1972}.

\begin{theorem}[Polydisk Theorem]\label{thm:polydisk} Suppose $\Omega$ is a bounded symmetric domain with real rank $r$. If  $z_1, z_2 \in \Omega$, then there exist a holomorphic isometric embedding $f:(\Db^r, K_{\Db^r}) \rightarrow (\Omega, K_\Omega)$ whose image contains $z_1, z_2$. \end{theorem}

For any bounded symmetric domain $\Omega \subset \Cb^d$, Harish-Chandra constructed an embedding $F:\Omega \hookrightarrow \Cb^d$ whose image is convex and bounded, see~\cite[Chapter II, Section 4]{S1980}. Further, there exists a norm $\norm{\cdot}_{HC}$ on $\Cb^d$ such that 
\begin{align*}
F(\Omega) = \left\{ z \in \Cb^d : \norm{z}_{HC} < 1\right\}.
\end{align*}
We will use the following terminology. 

\begin{definition} A bounded symmetric domain $\Omega \subset \Cb^d$ is in \emph{standard form} if it coincides with the image of its Harish-Chandra embedding. 
\end{definition}

We now recall the following well known description of the Bergman kernel on a bounded symmetric domain, see for instance~\cite{FK1990} or~\cite[Chapter II, Section 5]{S1980}. 

\begin{theorem}\label{thm:bergman_BSD} Suppose $\Omega \subset \Cb^d$ is a bounded symmetric domain in standard form with real rank $r$. Assume $\Phi:(\Db^r, K_{\Db^r}) \rightarrow (\Omega, K_\Omega)$ is a holomorphic isometric embedding with $\Phi(0)=0$. Then there exist constants $p,C > 0$ such that 
\begin{align*}
\kappa_\Omega( \Phi(z), \Phi(z) ) = C \left( \prod_{j=1}^r \left(1-\abs{z_j}^2 \right) \right)^{-p}
\end{align*}
for all $z \in \Db^r$. Moreover, $p \geq (d+r)/r$. 
\end{theorem}

Here are precise references for the proof of Theorem~\ref{thm:bergman_BSD}: by the discussion on pages 76 and 77 in~\cite{FK1990} there exist constants $p,C > 0$ such that 
\begin{align*}
\kappa_\Omega( \Phi(z), \Phi(z) ) = C \left( \prod_{j=1}^r \left(1-\abs{z_j}^2 \right) \right)^{-p}
\end{align*}
for all $z \in \Db^r$. The lower bound on $p$ follows from Equations (1.9) and (3.3) in~\cite{FK1990}.

\subsection{Complex geodesics in polydisks}

We will use the following observations about complex geodesics in polydisks. 

\begin{lemma}\label{lem:uniq_polydisk_1} Suppose $z=(z_1,\dots,z_r) \in \Db^r$ and 
\begin{align*}
 \abs{z_a} \neq \abs{z_b}
\end{align*}
for some $1 \leq a,b \leq r$. Then there exists two complex geodesics $\varphi_1, \varphi_2 : \Db \rightarrow \Db^r$ whose images contain $z$ and $0$, but $\varphi_1(\Db) \neq \varphi_2(\Db)$. 
\end{lemma}

\begin{proof} By permuting the coordinates we can assume that 
\begin{align*}
0 \leq \abs{z_1} \leq \abs{z_2} \leq \dots \leq \abs{z_r}.
\end{align*}
Since $\abs{z_1} < \abs{z_r}$ there exists two holomorphic functions $f_1, f_2 : \Db \rightarrow \Db$ such that $f_1(0)=f_2(0)=0$, $f_1(z_r) = f_2(z_r) = z_1$, and $f_1 \neq f_2$. For $2 \leq j \leq k-1$, select $\omega_j \in \overline{\Db}$ such that $\omega_j z_r = z_j$. Then for $j=1,2$, define the map
\begin{align*}
\varphi_j& : \Db \rightarrow \Db^r \\
\varphi_j&(\lambda) = (f_j(\lambda), \omega_2 \lambda, \dots, \omega_{r-1} \lambda, \lambda).
\end{align*}
Since each $\varphi_j$ is holomorphic, we have 
\begin{align*}
K_{\Db^r}(\varphi_j(\lambda_1), \varphi_j(\lambda_2)) \leq K_{\Db}(\lambda_1,\lambda_2)
\end{align*}
for all $\lambda_1,\lambda_2 \in \Db$. Further, by projecting onto the last component we have 
\begin{align*}
K_{\Db^r}(\varphi_j(\lambda_1), \varphi_j(\lambda_2)) \geq K_{\Db}(\lambda_1,\lambda_2)
\end{align*}
for all $\lambda_1, \lambda_2 \in \Db$. So $\varphi_1, \varphi_2 : \Db \rightarrow \Db^r$ are both complex geodesics. Finally, since $f_1 \neq f_2$, we have $\varphi_1(\Db) \neq \varphi_2(\Db)$. 
\end{proof}

\begin{lemma}\label{lem:uniq_polydisk_2} Suppose $z=(z_1,\dots,z_r) \in \Db^r$ and 
\begin{align*}
 0 < \abs{z_1} = \abs{z_2} = \dots = \abs{z_r}.
\end{align*}
If $\varphi : \Db \rightarrow \Db^r$ is a complex geodesic with $\varphi(0)=0$ and $\varphi(\lambda_0)=z$, then $\abs{\lambda_0} = \abs{z_1}$ and
\begin{align*}
\varphi(\lambda) = \left( \frac{z_1}{\lambda_0} \lambda, \dots, \frac{z_r}{\lambda_0} \lambda \right)
\end{align*}
for all $\lambda \in \Db$. 
\end{lemma}

\begin{proof} Since 
\begin{align*}
K_{\Db}(0,\lambda_0) = K_{\Db^r}(0,z) = \max_{1 \leq j \leq r} K_{\Db}(0,z_j) = K_{\Db}(0,z_1)
\end{align*}
we must have $\abs{\lambda_0} = \abs{z_1}$. Then applying the Schwarz lemma to each component function of $\varphi$ shows that 
\begin{align*}
\varphi(\lambda) = \left( \frac{z_1}{\lambda_0} \lambda, \dots, \frac{z_r}{\lambda_0} \lambda \right)
\end{align*}
for all $\lambda \in \Db$. 
\end{proof}

\section{Bounded symmetric domains with smooth boundaries}

\begin{theorem}\label{thm:BSD_reg} Suppose $\Omega$ is a bounded symmetric domain with $\Cc^{1,1}$ boundary. Then $\Omega$ is biholomorphic to the unit ball. \end{theorem}

The rest of the section is devoted to the proof of the Theorem. So suppose $\Omega$ is a bounded symmetric domain with $\Cc^{1,1}$ boundary.

Let $\Omega_{HC}$ denote the image of the Harish-Chandra embedding of $\Omega$. By Proposition~\ref{prop:rescaling}, there exists a biholomorphism $F:D \rightarrow \Omega_{HC} $ where $D \subset \Cb^d$ is a domain such that
\begin{align*}
\Pc_\alpha \subset D \subset \Pc_{\beta}
\end{align*}
for some $\alpha > \beta > 0$. By post-composing $F$ with an element of $\Aut(\Omega_{HC})$ we may assume that $F(e_1) = 0$.

\begin{lemma} There exists a holomorphic isometric embedding $\Phi:(\Db^r, K_{\Db^r}) \rightarrow (D, K_D)$ such that 
\begin{align*}
\Phi(\lambda,\dots,\lambda) = \left( \frac{1+\lambda}{1-\lambda},0,\dots,0 \right)
\end{align*}
for all $\lambda \in \Db$. 
\end{lemma}

\begin{proof}As before, define
\begin{align*}
\Hc_D := D \cap \Cb \cdot e_1 =  \left\{ (z,0,\dots, 0) \in \Cb^d : { \rm Re}(z) > 0\right\}.
\end{align*}

Fix $w_0 \in \Hc_D \setminus\{e_1\}$. By Theorem~\ref{thm:polydisk} there exists a holomorphic isometric embedding $\Phi_0:(\Db^r, K_{\Db^r}) \rightarrow (D, K_D)$ with $e_1, w_0 \in \Phi_0(\Db^r)$. By pre-composing $\Phi_0$ with an element of $\Aut(\Db^r) \geq \Aut(\Db) \times \cdots \times \Aut(\Db)$ we may assume that 
\begin{align*}
\Phi_0(0) = 0 \text{ and } \Phi_0(t_1,\dots,t_r) = w_0
\end{align*}
for some real numbers $t_1,\dots, t_r \in [0,1)$. 

By the ``in particular'' part of Proposition~\ref{prop:unique_slice}, every complex geodesic in $D$ containing $e_1,w_0$ has image $\Hc_D$. So by Lemma~\ref{lem:uniq_polydisk_1}  we must have
\begin{align*}
t_1 = \dots = t_r.
\end{align*}
Then by Lemma~\ref{lem:uniq_polydisk_2} 
\begin{align*}
\Hc_D = \{ \Phi_0(\lambda, \dots,\lambda) : \lambda \in \Db\}.
\end{align*}
Then by the first part of Proposition~\ref{prop:unique_slice}, there exists $\phi \in \Aut(\Db)$ such that 
\begin{align*}
\Phi_0(\phi(\lambda),\dots,\phi(\lambda)) = \left( \frac{1+\lambda}{1-\lambda},0,\dots,0 \right)
\end{align*}
for all $\lambda \in \Db$. Finally, the map $\Phi:=\Phi_0 \circ (\phi,\dots,\phi)$ has all of the desired properties. 
\end{proof} 

Next consider the function
\begin{align*}
J&:\Db \rightarrow \Cb \\
J&(\lambda) = \det\left(F^\prime\left( \frac{1+\lambda}{1-\lambda},0,\dots,0 \right)\right)=\det(F^\prime(\Phi(\lambda,\dots,\lambda)))
\end{align*}
where $F^\prime(z)$ is the complex Jacobian matrix of $F$. Since $F$ is a biholomorphism, $J$ is nowhere zero. We will show that $\Omega_{HC}$ is biholomorphic to the ball by estimating the boundary values of $J$. To that end, define $\Phi_{HC} := F \circ \Phi$. Notice that $\Phi_{HC}(0)= F(e_1)=0$. 

Let $\kappa_D$ and $\kappa_{\Omega_{HC}}$ be the Bergman kernels of $D$ and $\Omega_{HC}$ respectively. We will use the notation that 
\begin{align*}
\kappa_D(z) : = \kappa_D(z,z) \text{ and } \kappa_{\Omega_{HC}}(w) := \kappa_{\Omega_{HC}}(w,w)
\end{align*}
for $z \in D$ and $w \in \Omega_{HC}$.  Then, by Proposition~\ref{prop:change_of_variable},
\begin{align}
\label{eq:J_change_of_variable}
\abs{J(\lambda)}^2 =\frac{\kappa_D \left(  \Phi(\lambda,\dots,\lambda) \right)}{\kappa_{\Omega_{HC}} \left( \Phi_{HC}(\lambda,\dots,\lambda) \right)}= \frac{\kappa_D \left(  \frac{1+\lambda}{1-\lambda},0,\dots,0 \right)}{\kappa_{\Omega_{HC}} \left( \Phi_{HC}(\lambda,\dots,\lambda) \right)}
\end{align} 
for all $\lambda \in \Db$.

\begin{lemma}\label{lem:D_bergman_estimate} There exist constants $0 < a < b$ such that: 
\begin{align*}
a\left(\frac{1-\abs{\lambda}}{\abs{1-\lambda}^2} \right)^{-(d+1)} \leq \kappa_D \left(  \frac{1+\lambda}{1-\lambda},0,\dots,0 \right) \leq b \left(\frac{1-\abs{\lambda}}{\abs{1-\lambda}^2} \right)^{-(d+1)}
\end{align*}
for all $\lambda \in \Db$. 
\end{lemma}

\begin{proof} Since $\Omega \subset \Pc_\beta$, Proposition~\ref{prop:bergman_monotone} and Observation~\ref{obs:bergman_kernel_parabola}
 imply that there exists a constant  $C_\beta > 0$ such that 
\begin{align*}
C_\beta\left({\rm Re}(z)\right)^{-(d+1)} = K_{\Pc_\beta}(z,0,\dots,0) \leq K_D(z,0,\dots,0) 
\end{align*}
for all $z \in \Hc$. Further 
\begin{align*}
\frac{1-\abs{\lambda}}{\abs{1-\lambda}^2} \leq {\rm Re}\left( \frac{1+\lambda}{1-\lambda}\right) \leq 2\frac{1-\abs{\lambda}}{\abs{1-\lambda}^2}
\end{align*}
for all $\lambda \in \Db$. Combining these two estimates provides the lower bound:
\begin{align*}
C_\beta\left(\frac{1-\abs{\lambda}}{\abs{1-\lambda}^2} \right)^{-(d+1)} \leq \kappa_D \left(  \frac{1+\lambda}{1-\lambda},0,\dots,0 \right)
\end{align*}
The same argument with $\Pc_\alpha \subset \Omega$ yields the upper bound. 
\end{proof}

\begin{lemma} $\Omega$ is biholomorphic to the unit ball. \end{lemma}

\begin{proof}
By Equation~\eqref{eq:J_change_of_variable}, Theorem~\ref{thm:bergman_BSD}, and Lemma~\ref{lem:D_bergman_estimate} there exist constants $C>0$ and $p\geq (d+r)/r$ such that
\begin{align*}
\abs{J(\lambda)}^2 & \leq C\left(\frac{1-\abs{\lambda}}{\abs{1-\lambda}^2} \right)^{-(d+1)}\left( 1-\abs{\lambda}^2 \right)^{rp}  \leq C\left(\frac{1-\abs{\lambda}}{\abs{1-\lambda}^2} \right)^{-(d+1)}\left( 1-\abs{\lambda}^2 \right)^{d+r} \\
& \leq C \abs{1-\lambda}^{2(d+1)} (1+\abs{\lambda})^{d+r} (1-\abs{\lambda})^{r-1}
\end{align*}
for all $\lambda \in \Db$. 

So if $r > 1$, then $J$ extends continuously to $\partial \Db$ and $J|_{\partial \Db} \equiv 0$. Then by the maximal principle, $J \equiv 0$. But this contradicts the fact that $J$ is nowhere vanishing. So we must have $r=1$ and hence $\Omega$ is biholomorphic to the unit ball. 

\end{proof}
 
\section{Proof of Theorem~\ref{thm:main}}

Theorem~\ref{thm:main} is an immediate consequence of Proposition~\ref{prop:BSD} and Theorem~\ref{thm:BSD_reg}.

\bibliographystyle{alpha}
\bibliography{complex_kob}

\end{document}